\documentclass[reqno]{amsart} 
\usepackage{amsfonts, amsmath, amsthm, amssymb, latexsym, graphicx,color}

\renewcommand{\tilde}{\widetilde}

\def\cal{\mathcal }
\theoremstyle{plain}
\newtheorem{theorem}{Theorem}[section]
\newtheorem{corollary}[theorem]{Corollary}
\newtheorem{lemma}[theorem]{Lemma}
\newtheorem{proposition}[theorem]{Proposition}
\theoremstyle{definition}

\newtheorem{remark}[theorem]{Remark}

\numberwithin{equation}{section}

\title[Martingales associated with Bessel processes]{Some martingales associated with
  multivariate Bessel  processes}

\author{Miklos Kornyik}
\address{E\"otv\"os Lor\'and University, Department of Probability Theory and Statistics,
P\'azm\'any P\'eter s\'et\'any 1/C., H-1117, Budapest, Hungary \and
Wigner Research Centre for Physics, Department of Quantum Optics and Quantum Information,  Konkoly-Thege Mikl\'os \'ut 29-33., H-1121, Budapest, Hungary }
\email{koma@cs.elte.hu}
\author{ Michael Voit}
\address{Fakult\"at Mathematik, Technische Universit\"at Dortmund,
          Vogelpothsweg 87,
          D-44221 Dortmund, Germany}
\email{michael.voit@math.tu-dortmund.de}
\author{ Jeannette H.C. Woerner} 
\address{Fakult\"at Mathematik, Technische Universit\"at Dortmund,
          Vogelpothsweg 87,
          D-44221 Dortmund, Germany}
\email{jeannette.woerner@math.tu-dortmund.de}


\begin{document}
\date{\today}

\begin{abstract} We study  Bessel processes on Weyl chambers of types A and B on
$\mathbb R^N$. Using elementary symmetric functions, we present several space-time-harmonic functions and thus
martingales for these processes $(X_t)_{t\ge0}$ which are independent from one parameter of these processes.
As a consequence, $p(y):=\mathbb E(\prod_{i=1}^N (y-X_t^i))$ can be expressed via classical 
orthogonal polynomials. Such formulas on characteristic polynomials admit interpretations
 in random matrix theory where they are partially known by Diaconis, Forrester, and Gamburd.
\end{abstract}

\maketitle

\section{Introduction}

Interacting Calogero-Moser-Sutherland particle models
 on  $\mathbb R$ with $N$ particles 
 can be described via Bessel processes $(X_t)_{t\ge0}$ associated with root systems;
 see e.g.~\cite{CGY,GY,R, RV1,RV2,DV,AKM1, AKM2}.
These processes are classified via root systems and finitely many  multiplicity parameters
which control the interaction. 
In this paper, we only study the root systems
 $A_{N-1}$ and $B_N$.

In the  case $A_{N-1}$, we have a  multiplicity $k\in[0,\infty[$,
 $(X_t)_{t\ge0}$ lives on the closed Weyl chamber
$$C_N^A:=\{x\in \mathbb R^N: \quad x_1\ge x_2\ge\ldots\ge x_N\},$$
 the generator of the transition semigroup is 
\begin{equation}\label{def-L-A} \Delta_k f:= \frac{1}{2} \Delta f +
 k \sum_{i=1}^N\Bigl( \sum_{j:j\ne i} \frac{1}{x_i-x_j}\Bigr) \frac{\partial}{\partial x_i}f ,
 \end{equation}
and we assume reflecting boundaries, i.e., the domain of $ \Delta_k$ may be chosen as
$$D(\Delta_k):=\{f|_{C_N^A}: \>\> f\in C^{(2)}(\mathbb R^N), \>\>\> f\>\>\text{ invariant under all coordinate-premutations}\}.$$

In the case $B_N$,  we have two  multiplicities $k_1,k_2\ge 0$, the processes live on 
$$C_N^B:=\{x\in \mathbb R^N: \quad x_1\ge x_2\ge\ldots\ge x_N\ge0\},$$
 the generator of the transition semigroup is 
\begin{equation}\label{def-L-B} \Delta_{k_1,k_2}f:= \frac{1}{2} \Delta f +
 k_2 \sum_{i=1}^N \sum_{j:j\ne i} \Bigl( \frac{1}{x_i-x_j}+\frac{1}{x_i+x_j}  \Bigr)
 \frac{\partial}{\partial x_i}f 
\> + k_1\sum_{i=1}^N\frac{1}{x_i}\frac{\partial}{\partial x_i}f, \end{equation}
and we again assume reflecting boundaries.

The transition probabilities of the  diffusions
$(X_{t})_{t\ge0}$ on $C_N$ (with $C_N=C_N^A$ or $C_N^B$) are  as
 follows by \cite{R,RV1,RV2}:
For any $x\in C_N$, and $E\subset C_N$ Borel set, 
\begin{equation}\label{density-general}
K_t(x,E)=c_k \int_E t^{-\gamma_k-N/2}\  \mathrm e^{-(\|x\|^2+\|y\|^2)/(2t)} J_k(\frac{x}{\sqrt{t}}, \frac{y}{\sqrt{t}}) 
\cdot w_k(y)\> dy
\end{equation}
with
\begin{equation}\label{def-wk}
w_k^A(x):= \prod_{i<j}(x_i-x_j)^{2k}, \quad
w_k^B(x):= \prod_{i<j}(x_i^2-x_j^2)^{2k_2}\cdot \prod_{i=1}^N x_i^{2k_1},\end{equation}
and
\begin{equation}\label{def-gamma}
\gamma_k^A=kN(N-1)/2, \quad\quad  \gamma_{(k_1,k_2)}^B=k_2N(N-1)+k_1N
\end{equation}
respectively. The weights
$w_k$ are homogeneous of degree $2\gamma_k$.
Furthermore, the
$c_k>0$ are known normalization constants, and
$J_k$ is a multivariate Bessel function of
 type $A_{N-1}$ or $B_N$ with multiplicities $k$ or $(k_1,k_2)$ respectively;
 see e.g. \cite{R}.
$J_k$ is analytic on $\mathbb C^N \times \mathbb C^N $ with
$ J_k(x,y)>0$ for $x,y\in \mathbb R^N $.
Moreover, $J_k(x,y)=J_k(y,x)$ and $J_k(0,y)=1$
for all $x,y\in \mathbb C^N $.

Therefore, if $X_0=0$, then for $t>0$, $X_{t}$ has the Lebesgue density
\begin{equation}\label{density-A-0}
 \frac{c_k}{t^{\gamma_k+N/2}} e^{-\|y\|^2/(2t)} \cdot w_k(y)\> dy
\end{equation}
on $C_N$ for $t>0$.
 In particular, in the case $A_{N-1}$ for $k=1/2, 1,2$,  $X_{t}$
has the distribution of the eigenvalues of
Gaussian orthogonal, unitary, and symplectic ensembles up to scalings; see e.g. \cite{Me}.
Moreover, for general $k>0$, these distributions appear as spectral  distributions of
  the tridiagonal $\beta$-Hermite ensembles of
Dumitriu and Edelman \cite{DE1}. Similar interpretations exist in the  case $B_{N}$
for Laguerre ensembles and the  tridiagonal $\beta$-Laguerre ensembles of \cite{DE1}.

In this paper we use the generators (\ref{def-L-A}), (\ref{def-L-B}) and the associated 
stochastic differential equations for the diffusions $(X_t)_{t\ge0}$ and construct polynomials $p_l$ 
in $N+1$ variables of order $l=1,\ldots,N$ such that the processes $(p_l(X_t,t))_{t\ge0}$
are martingales. The functions $p_l$
 are constructed via elementary symmetric polynomials such that,
 up to some rescaling of the processes  $(X_t)_{t\ge0}$, these polynomials do not depend on the parameter  $k$ in the
 case $A_{N-1}$, while they depend on one parameter only in the case $B_N$. Moreover, due to this observation,
 our martingale result can be also extended to the case where the independent parameter is equal to $\infty$,
in which case the SDE of the renormalization of $(X_t)_{t\ge0}$ simplifies to an ODE. The solution of this ODE starting in the origin $0\in C_N$ can be described explicitely in terms of the zeros of the Hermite polynomial 
$H_N$ or the Laguerre polynomial $L_N^{(\alpha)}$ with a suitable $\alpha$  in the cases $A_{N-1}$ or $B_N$
respectively.
This observation will lead to closed formulas for 
$$\mathbb E\bigl(\prod_{i=1}^N (y- X_{t,i}^\beta)\bigr) \quad (y\in\mathbb R)$$
for the processes $(X_t)_{t\ge0}$ starting in $0$, which involve  Hermite and  Laguerre polynomials.
These formulas are known for the multiplicities associated with the classical random matrix ensembles 
\cite{DG, FG}. 

This paper is organized as follows: In Section 2 we consider the case $A_{N-1}$ while Section 3 is devoted to
 the case $B_N$.

We finally recapitulate  the following well-known result
 (see Lemma 3.4, Corollary 6.6, and Proposition 6.8 of \cite{CGY} and
 \cite{Sch, GM} for corrections of the proofs of these results) which will be the basis 
for our SDE approach:

\begin{theorem}\label{SDE-basic} Let $k>0$ or $k_1,k_2>0$ in the 
$A_{N-1}$- or $B_N$-case. Then, for each starting point $x\in C_N$ and  $t>0$,
 the initial problem
\begin{equation}\label{SDE-general}
X_0=x,\quad\quad dX_t= dB_t +  \frac{1}{2} (\nabla(\ln w_k))(X_t) \> dt
\end{equation}
 where $w_k$ is defined by (\ref{def-wk}) and $(B_t)_{t\ge0}$ is an $N$-dimensional Brownian motion,
has a unique strong solution $(X_t)_{t\ge0}$.
This solution is a Bessel process as  above.

Moreover, if  $k\ge 1/2$ in the 
$A_{N-1}$-case or $k_1,k_2\ge 1/2$ in the $B_N$-case,  then $X_t$ is in  the interior on $C_N$
 almost surely for $t>0$.
\end{theorem}

\section{Bessel processes of type A} 

We now study  Bessel processes
of type $A_{N-1}$ where we denote the multiplicities by  $\beta\ge 0$ instead of $k$
 in order to avoid  confusions with indices and coordinates.
Therefore, let  $( X_{t}^\beta)_{t\ge0}$ be a  Bessel process
of type $A_{N-1}$  with multiplicity $\beta\ge 0$ with values from
$$C_N^A:=\{x\in \mathbb R^N: \quad x_1\ge x_2\ge\ldots\ge x_N\}.$$
 $( X_{t}^\beta)_{t\ge0}$  satisfies the SDE 
\begin{equation}\label{SDE-A}
 dX_{t,i}^\beta = dB_{t,i}+ \beta\sum_{j:j\ne i} \frac{1}{X_{t,i}^\beta-X_{t,j}^\beta} dt \quad\quad(i=1,\ldots,N).
\end{equation}
with an $N$-dimensional Brownian motion $(B_{t,1},\ldots,B_{t,N})_{t\ge0}$, 
where the paths are reflected whenever they hit the boundary $\partial C_N^A$ of $C_N^A$.
By Theorem \ref{SDE-basic},  $( X_{t}^\beta)_{t\ge0}$ 
 does not meet  $\partial C_N^A$ for $t>0$ a.s.~for $\beta\ge 1/2$.

For the sake of convenience we will study the renormalized processes
 $(\tilde X_{t}:=\tilde X_{t}^\beta:=X_{t}^ \beta/\sqrt \beta)_{t\ge0}$ which satisfy
\begin{equation}\label{SDE-A-normalized}
d\tilde X_{t,i}^\beta =\frac{1}{\sqrt \beta}dB_{t,i} + \sum_{j:j\ne i} 
 \frac{1}{\tilde X_{t,i}^\beta-\tilde X_{t,j}^\beta }dt\quad\quad(i=1,\ldots,m). 
\end{equation}
>From now on the parameter $\beta$ in $\tilde X^\beta_t$ will be omitted unless it is explicitly referred to. 

Now we will derive some results for symmetric polynomials of $( X_{t})_{t\ge0}$ and
$(\tilde X_{t})_{t\ge0}$.
First let us we recapitulate that the elementary
symmetric polynomials $e_k^m$ ($m\in \mathbb N,$ $ k=0,\ldots,m$) in $m$ variables  are characterized by
\begin{equation}\label{symmetric-poly}
\prod_{k=1}^m (z-x_k) = \sum_{k=0}^{m}(-1)^{m-k}  e^m_{m-k}(x) z^k \quad\quad (z\in\mathbb C, \> x=(x_1,\ldots,x_m)),
\end{equation}
in particular,
$e_0^m=1, \> e_1^m(x)=\sum_{k=1}^m x_k ,\quad \ldots\quad  e_m^m(x)=\prod_{k=1}^m x_k$. 

We need a further notation: For  a non-empty set $S\subset \{1,\ldots,N\}$, let $\tilde X_{t,S}$ be the $\mathbb R^{|S|}$-valued
 random vector with the 
coordinates $\tilde  X_{t,i}$ for $i\in S$ in the natural ordering on $S$. 
 The following technical observation is our starting point:

\begin{lemma}\label{symmetric-pol-in-t}
 For $\beta\ge 1/2$,  $k=2,\ldots,N$, and $l\ge 0$
\begin{align}d(t^l\cdot e_k^N(\tilde X_{t}))=&
\frac{t^l}{\sqrt\beta}\sum_{j=1}^N e_{k-1}^{N-1}(\tilde X_{t,\{1,\ldots,N\}\setminus\{j\}})\>
 dB_{t,j}\notag\\
&+\Biggl(  lt^{l-1}\cdot e^N_k(\tilde X_{t})
- \frac{t^l}{2} (N-k+2)(N-k+1)e_{k-2}^{N}(\tilde X_{t})\Biggr)dt.
\notag\end{align}
\end{lemma}

\begin{proof}
 It\^{o}'s formula and the SDE  (\ref{SDE-A-normalized}) show that
$$d(t^l\cdot e_k^N(\tilde X_{t}))= 
lt^{l-1}\cdot e_k^N(\tilde X_{t}) \> dt+
t^l\sum_{j=1}^N  e_{k-1}^{N-1}(\tilde X_{t,\{1,\ldots,N\}\setminus\{j\}})\> d\tilde X_{t,j}.$$
Therefore, by the SDE (\ref{SDE-A-normalized}), 
\begin{align}\label{elementary-symm-a-1}
d(t^l \cdot e_k^N(\tilde X_{t}))&=  lt^{l-1}\cdot e_k^N(\tilde X_{t}) \> dt+ \frac{t^l}{\sqrt\beta}
\sum_{j=1}^N  e_{k-1}^{N-1}(\tilde X_{t,\{1,\ldots,N\}\setminus\{j\}})\>
 dB_{t,j}\\
&\quad+
t^l\sum_{j=1}^N \sum_{i: i\ne j} 
\frac{  e_{k-1}^{N-1}(\tilde X_{t,\{1,\ldots,N\}\setminus\{j\}})}{\tilde X_{t,j}-\tilde X_{t,i}}dt\notag\\
&=  lt^{l-1}\cdot e_k^N(\tilde X_{t}) \> dt+
\frac{t^l}{\sqrt\beta}\sum_{j=1}^N e_{k-1}^{N-1}(\tilde X_{t\{1,\ldots,N\}\setminus\{j\}})\>
 dB_{t,j}\notag\\
&\quad+
\frac{t^l}{2}     \sum_{i,j=1,\ldots,n; i\ne j}
\frac{  e_{k-1}^{N-1}(\tilde X_{t,\{1,\ldots,N\}\setminus\{j\}})- 
e_{k-1}^{N-1}(\tilde X_{t,\{1,\ldots,N\}\setminus\{i\}})  }{\tilde X_{t,j}-\tilde X_{t,i}}dt.\notag
\end{align}
Moreover, by simple combinatorial computations (see (2.10), (2.11) in \cite{VW}) we have for $i\neq j$ that
\begin{equation}\label{elementary-symm-a-2}
  e_{k-1}^{N-1}(\tilde X_{t,\{1,\ldots,N\}\setminus\{j\}})- e_{k-1}^{N-1}(\tilde X_{t,\{1,\ldots,N\}\setminus\{i\}})=
 (\tilde X_{t,i}-\tilde X_{t,j})e_{k-2}^{N-2}(\tilde X_{t,\{1,\ldots,N\}\setminus\{i,j\}})
\end{equation}
and 
\begin{equation}\label{elementary-symm-a-3}
 \sum_{i,j=1,\ldots,N; i\ne j}e_{k-2}^{N-2}(\tilde X_{t,\{1,\ldots,N\}\setminus\{i,j\}})= 
(N-k+2)(N-k+1)e^N_{k-2}(\tilde X_{t}).
\end{equation}
 (\ref{elementary-symm-a-1})-(\ref{elementary-symm-a-3}) now lead to the lemma.
\end{proof}

We also need the following well known observation.
 Here we always use the canonical filtration of the Brownian motion $(B_t)_{t\ge0}$.

\begin{lemma}\label{Brownian-martingale}
For each polynomial $p$ in $N$ variables, $\beta\ge0$, $i=1,\ldots,N$, and  $l\ge0$, the process
$\Bigl(\int_0^t s^l\cdot p(\tilde X_{s})\> dB_{s,i})_{t\ge0}$
is a martingale.
\end{lemma}

\begin{proof} The process $(\sum_{i=1}^N  X_{t,i}^2)_{t\ge0}$ is a classical one-dimensional squared Bessel process
(see e.g. \cite{RV1}), i.e., all powers of this process are square-integrable w.r.t.~the measure
 ${\cal P}\otimes\lambda|_{[0,t]}$ on $\Omega\times[0,t]$ for each $t>0$, the probability measure ${\cal P}$
 on the underlying probability space $\Omega$, and the Lebesgue measure $\lambda$.
The lemma is now clear by the very construction of the It\^{o} integral.
\end{proof}

Using Lemmas \ref{symmetric-pol-in-t}   and \ref{Brownian-martingale}, the following martingales can be constructed
for Bessel processes via elementary symmetric polynomials:

\begin{proposition}\label{elementary-symm-martingale}
For all $\beta>0$, $k=1,\ldots,N$, and all starting points $x_0\in C_N^A$ of $(\tilde X_{t,\beta})_{t\ge0}$,
 the process
\begin{equation}\label{mart-formel-a}
\Biggl( e^N_k(\tilde X_{t}) + \sum_{l=1}^{\lfloor k/2\rfloor} \frac{(N-k+2l)!}{2^l\cdot l!\cdot (N-k)!} 
t^l\cdot e^N_{k-2l}(\tilde X_{t})\Biggr)_{t\ge0}
\end{equation}
is a martingale.
\end{proposition}

\begin{proof} 
First assume  $\beta\ge1/2$. For $k=1$ we obtain from the SDE (\ref{SDE-A-normalized}) that
\begin{equation}\label{e-1-brownian}
 e^N_{1}(\tilde X_{s})=\sum_{i=1}^N\tilde X_{s,i} 
=\frac{1}{\sqrt \beta}\sum_{i=1}^N B_{s,i} + e^N_1(x_0)\end{equation}
which proves the proposition for $k=1$. In general,
it follows  from  Lemmas \ref{symmetric-pol-in-t}   and \ref{Brownian-martingale} that
\begin{align}\label{mart-long}
\Biggl( e^N_k(\tilde X_{t}) +& \sum_{l=1}^{\lfloor k/2\rfloor-1} \frac{(N-k+2l)!}{2^l\cdot l!\cdot (N-k)!} 
t^l\cdot e^N_{k-2l}(\tilde X_{t}) \\& + \frac{(N-k+2 \lfloor k/2\rfloor  )!}{2^{\lfloor k/2\rfloor }\cdot
 (\lfloor k/2\rfloor-1)!\cdot (N-k)!} \int_0^ts^{\lfloor k/2\rfloor-1 }\cdot e^N_{k-2\lfloor k/2\rfloor}(\tilde X_{s}) \> ds
\Biggr)_{t\ge0}
\notag\end{align}
is  a martingale. We thus only have to compare the last term on the RHS of (\ref{mart-long}) with the summand
 $l=\lfloor k/2\rfloor$  on the RHS of (\ref{mart-formel-a}).
Here the arguments are different for even and odd $k$.

If $k$ is even, then $ e^N_{k-2\lfloor k/2\rfloor}=e^N_0=1$ and
$\int_0^ts^{\lfloor k/2\rfloor-1 } \> ds=\frac {1}{\lfloor k/2\rfloor } t^{\lfloor k/2\rfloor }.$
This shows that the last term  on the RHS of (\ref{mart-long})
 is the last summand for $l= \lfloor k/2\rfloor$ in (\ref{mart-formel-a}).
 This  yields the claim in the even case.

Now let $k\ge3$ be odd. Here $e^N_{k-2\lfloor k/2\rfloor}=e^N_1$, and we
 obtain from the SDE (\ref{SDE-A-normalized}) and It\^{o}'s formula that
\begin{align} 
t^{\lfloor k/2\rfloor}& e^N_{k-2\lfloor k/2\rfloor}(\tilde X_{t}) =
 t^{\lfloor k/2\rfloor} \sum_{i=1}^N\tilde X_{t,i} = 
\frac{ t^{\lfloor k/2\rfloor}}{\sqrt \beta}\sum_{i=1}^N B_{t,i} \> +\>  t^{\lfloor k/2\rfloor}e_1^N(x_0)\notag
\\
&=\frac{1}{\sqrt \beta}\int_0^t s^{\lfloor k/2\rfloor }\> d\Bigl( \sum_{i=1}^N B_{s,i}\Bigr) 
+\frac{ \lfloor k/2\rfloor}{\sqrt \beta} \int_0^t\Bigl( \sum_{i=1}^N B_{s,i}\Bigr)s^{\lfloor k/2\rfloor-1 } \> ds
+ t^{\lfloor k/2\rfloor}e^N_1(x_0)\notag
\\
&=\frac{1}{\sqrt \beta}\int_0^t s^{\lfloor k/2\rfloor }\> d\Bigl( \sum_{i=1}^N B_{s,i}\Bigr) \notag\\
&\quad\quad\quad
+ \lfloor k/2\rfloor \int_0^t(e^N_{1}(\tilde X_{s})-e^N_1(x_0))s^{\lfloor k/2\rfloor-1 }
 \> ds
+ t^{\lfloor k/2\rfloor}e^N_1(x_0)\notag
\\
&=\frac{1}{\sqrt \beta}\int_0^t s^{\lfloor k/2\rfloor }\> d\Bigl( \sum_{i=1}^N B_{s,i}\Bigr) +
 \lfloor k/2\rfloor \int_0^te^N_{1}(\tilde X_{s})s^{\lfloor k/2\rfloor-1 }
 \> ds.
\notag\end{align}
This and  (\ref{mart-long}) yield the proposition for $k\ge3$ odd.
In summary, the proposition holds for $\beta\ge1/2$.

We now use Dynkin's formula (see e.g. Section III.10 of \cite{RW}) which implies that
 the symmetric functions
 $$  f_{N,k}:C_N^A\times [0,\infty[\to \mathbb R, \quad (x,t)\mapsto
 e^N_k(x) + \sum_{l=1}^{\lfloor k/2\rfloor} \frac{(N-k+2l)!}{2^l\cdot l!\cdot (N-k)!} 
t^l\cdot e^N_{k-2l}(x)$$ 
are space-time-harmonic w.r.t. the generators
$$\tilde\Delta_\beta f= \frac{1}{2\sqrt\beta}\Delta f +
 \sum_{i=1}^N\Bigl( \sum_{j\ne i} \frac{1}{x_i-x_j}\Bigr) \frac{\partial}{\partial x_i}f$$
 of 
 the diffusions $(\tilde X_t)_{t\ge0}$   for $\beta\ge1/2$, i.e., we have
\begin{equation}\label{diff-op-space-time}
(\frac{\partial}{\partial t}+ \tilde\Delta_\beta)f_{N,k}\equiv 0.
\end{equation}
 As the left hand side of (\ref{diff-op-space-time}) is analytic in $\beta$,
  analytic continuation shows that $f_{N,k}$ is  space-time-harmonic
 also  for all  $\beta>0$. Dynkin's formula  now yields the
 proposition in general.
\end{proof}

Notice that the functions $f_{N,k}$  do
 not depend on $\beta>0$, and that the simultanous space-time harmonicity w.r.t.~ all $\tilde\Delta_\beta$
 ($\beta>0$) is trivial, as $\beta$ only appears as a factor of the classical Laplacian $\Delta$,
 for which obviously $\Delta e^N_k\equiv$ holds for all $k$.

We also point out that
 Lemma \ref{symmetric-pol-in-t} and
 Proposition \ref{elementary-symm-martingale} remain valid for $\beta=\infty$, in which case the
 SDE (\ref{SDE-A-normalized}) is an ODE, and the process $(\tilde X_{t}^\infty)_{t\ge0}$ is deterministic
 whenever so is the initial condition for $t=0$. There are several limit theorems 
(laws of large numbers, CLTs) for the limit transition $\beta\to\infty$; see \cite{AKM1, AV, VW}. 
  Proposition \ref{elementary-symm-martingale} for $\beta\in]0,\infty]$ leads to:

\begin{corollary}\label{constant-expectation}
 Let $(\tilde X_{t}^\beta)_{t\ge0}$ be a normalized Bessel process for $\beta\in]0,\infty]$
which starts in some  $x_0\in C_N^A$.
Then for $k=0,1,\ldots,N$ and $t\ge0$,  the expectations
$\mathbb E( e_k^N(\tilde X_{t}^\beta) )$ do not depend on
 $\beta$.
\end{corollary}

\begin{proof} 
The the case $k=0$ is trivial, and, by (\ref{e-1-brownian}),
 $\mathbb E( e_1^N(\tilde X_{t}^\beta) )=0$, which proves the result for $k=1$.
 Proposition \ref{elementary-symm-martingale} and induction now lead to the general case. 
\end{proof}

We now study the case when the process is initially in the origin. For $\beta=\infty$, the solution of the
 ODE (\ref{SDE-A-normalized}) can be expressed via the ordered zeros $z_1> z_2>\ldots>z_N$ of
 the Hermite polynomial 
\begin{equation}\label{def-hermite}
H_N(x)= \sum_{k=0}^{\lfloor k/2\rfloor} (-1)^k \frac{N!}{k! \> (N-2k)!} 2^{N-2k}x^{N-2k}
\end{equation}
 where the Hermite polynomials $(H_N)_{N\ge 0}$ are orthogonal w.r.t.
 the density  $e^{-x^2}$; see  \cite{S} for details. We have the following result by \cite{AV}
 which follows easily from Section 6.7 of \cite{S} on the zeros of $H_N$:

\begin{lemma}\label{special-solution}
The solution of the
 ODE (\ref{SDE-A-normalized}) with $\beta=\infty$ and start in $0\in C_N^A$
is given by $\tilde X_{t}^\infty= \sqrt{2t}\cdot {z} $ with $z:=(z_1,\ldots,z_N)$.
\end{lemma}

This and the preceding results have the following consequence:

\begin{corollary}\label{det-formula} For $\beta\in]0,\infty[$ let
$( X_{t,\beta})_{t\ge0}$ be the Bessel process of type A with  start in $0$.
Then,
\begin{align}\label{det-form-1}
\mathbb E\bigl(\prod_{i=1}^N (y- X_{t,i}^\beta)\bigr) & =  
\nonumber  c_\beta^A \int_{\mathbb R^N}\big( \prod_{i=1}^N (y-x_i)\big)\cdot  t^{-\gamma_\beta^A- N/2 } \mathrm{e}^{-\|x\|^2/(2t)} \prod_{i<j} (x_i-x_j)^{2\beta}\ dx \\
 &  = (t\beta/2)^{N/2}\cdot H_N(y/ \sqrt{2\beta t})
  \quad\quad\text{for}\quad y\in \mathbb R.
\end{align}
Moreover, for $k=0,1,\ldots,\lfloor k/2\rfloor-1$, $\mathbb E\bigl(e^N_{2k+1}(X_{t}^\beta)\bigr)=0$, and
\begin{equation}\label{det-form-2}
 \mathbb E\bigl(e_{2k}^N(X_{t}^\beta)\bigr)=   (t\beta/2)^k \frac{N!}{k! \> (N-2k)!} 
\quad\quad(k=0,1,\ldots,\lfloor k/2\rfloor).
\end{equation}
\end{corollary}

\begin{proof} Corollary \ref{constant-expectation} and Lemma \ref{special-solution} yield
\begin{align}\label{det-computation}
&\mathbb E\bigl(\prod_{i=1}^N (y- X_{t,i}^\beta)\bigr)=
\sum_{k=0}^N (-1)^k\mathbb E\bigl(e_k^N( X_{t}^\beta)\bigr)\cdot y^{N-k}\\
&=\beta^{N/2}\sum_{k=0}^N (-1)^k \mathbb E\bigl(e_k^N(\tilde X_{t}^\beta)\bigr) (y/\sqrt\beta)^{N-k}\notag\\
&=\beta^{N/2}\sum_{k=0}^N (-1)^k\mathbb E\bigl(e^N_k(\tilde X_{t}^\infty)\bigr)(y/\sqrt\beta)^{N-k}\notag\\
&=\beta^{N/2}\sum_{k=0}^N (-1)^k  e_k^N(\sqrt{2t}\cdot {z})(y/\sqrt\beta)^{N-k}\notag\\
&=(2t\beta)^{N/2}\sum_{k=0}^N (-1)^k e_k^N( {z})(y/\sqrt{2t\beta})^{N-k}
\notag\\
&=(2t\beta)^{N/2}\prod_{i=1}^N ( y/\sqrt{2t\beta}- z_i)
=(2t\beta)^{N/2}\frac{1}{2^N}\cdot  H_N(y/ \sqrt{2t\beta }) .
\notag\end{align}
This proves the first statement. The second one follows by a comparison of the coefficients in 
(\ref{def-hermite}) and (\ref{det-computation}).
\end{proof}

\begin{remark}
For  $\beta=1/2,1,2$ and start in $0$, the random variables $X_{t}^\beta$ have the same 
distributions as the ordered eigenvalues of a Gaussian orthogonal,
 unitary, or symplectic ensemble processes respectively
up to normalizations by (\ref{density-A-0}).
In this way, Corollary \ref{det-formula} can be restated for these ensembles. In particular, (\ref{det-form-1}) 
yields Proposition 11 of \cite{FG}
in the Gaussian unitary case; see also \cite{DG}. 
\end{remark}

\begin{remark} By \cite{AV}, the solution of the ODE (\ref{SDE-A-normalized}) with $\beta=\infty$ and start in $cz$, $c\geq 0$, is given by $\sqrt{2t +c^2}z$. For Bessel processes of type A with start in $cz$, this leads to
$$\mathbb E\bigl(\prod_{i=1}^N (y- X_{t,i}^\beta)\bigr) = ((2t+c^2)\beta/4)^{N/2}\cdot H_N(y/ \sqrt{(2t+c^2)\beta }).$$ 
\end{remark}

\begin{remark}
All preceding results are concerned with formulas which are invariant under the
 canonical action of the symmetric group $S_N$ on $\mathbb R^N$. We thus can replace the Bessel processes 
$( X_{t}^\beta)_{t\ge0}$  by Dunkl processes of type $A_{N-1}$
in Proposition \ref{elementary-symm-martingale} 
  and Corollaries \ref{constant-expectation} and \ref{det-formula}.
For the theory of Dunkl processes we refer to \cite{CGY}, \cite{GY}, \cite{RV1}, \cite{RV2}. 
 \end{remark}

\begin{remark} Corollary \ref{det-formula} is also valid for the case $\beta=0$. Here, the Dunkl process is an
$N$-dimensional Brownian motion, and the Bessel process a  Brownian motion on $C_N^A$
which is reflected on $\partial C_N^A$. We here have $\mathbb E(e^N_k(X_{t}^0))=0$ for $k\ge1$.
\end{remark}

\section{Bessel processes of type B} 

In this section we study Bessel processes 
for the root systems  $B_N$ with multiplicities $(k_1,k_2):=(\nu\cdot\beta, \beta)$
with parameters $\nu\ge0,\beta>0$.  These processes $( X_{t}:=(X_{t,1},\ldots,X_{t,N} ))_{t\ge0}$
have values on the closed Weyl chamber
$$C_N^B:=\{x\in \mathbb R^N: \quad x_1\ge x_2\ge\ldots\ge x_N\ge0\}$$
of type B and satisfy the  SDE 
\begin{equation}\label{SDE-B}
 dX_{t,i}^\beta = dB_{t,i}+ \beta \sum_{j\ne i} \Bigl(\frac{1}{X^\beta_{t,i}-X^\beta_{t,j}}+ \frac{1}{X^\beta_{t,i}+X^\beta_{t,j}}\Bigr)dt
+ \frac{\nu\cdot\beta}{X^\beta_{t,i}}dt\end{equation}
for $i=1,\ldots,N$
with an $N$-dimensional Brownian motion $(B_t)_{t\ge0}$
where the paths are reflected when they meet the boundary $\partial C_N^B$ of $C_N^B$.
Again, by Proposition 6.1 of \cite{CGY}, 
the process does not meet the boundary in positive time almost surely for $\beta\ge 1/2$ and $\nu\ge1$.

Again we  also  study the  renormalized processes
 $(\tilde X^\beta_{t}:=X_{t}/\sqrt \beta)_{t\ge0}$ with
\begin{equation}\label{SDE-B-normalized}
d\tilde X^\beta_{t,i} =\frac{1}{\sqrt \beta}dB_{t,i} + \sum_{j\ne i} 
 \Bigl( \frac{1}{\tilde X^\beta_{t,i}-\tilde X^\beta_{t,j}}
+ \frac{1}{\tilde X^\beta_{t,i}+\tilde X^\beta_{t,j}}\Bigr)dt +\frac{\nu}{\tilde X^\beta_{t,i}}dt
\end{equation}
for  $i=1,\ldots,N$ where we usually omit the parameter $\beta$ in $\tilde X_t^\beta$.
We now derive an analogue of Lemma \ref{symmetric-pol-in-t} which involves functions, which are invariant under
the Weyl group of type $B_N$. As in the proof of Lemma \ref{symmetric-pol-in-t}, let
$\tilde X_{t,S}$ be the $\mathbb R^{|S|}$-valued
 random vector with the 
coordinates $\tilde  X_{t,i}$ for $i\in S$ in the natural ordering on a subset $S\subset\{1,\ldots,N\}$.

\begin{lemma}\label{symmetric-pol-in-t-b}
 For all $\beta\ge 1/2$, $\nu\ge1$, $k=1,2,\ldots,N$, and $l\ge 0$
\begin{align}d(t^l\cdot e_k^N(\tilde X_{t}^2)&)=
\frac{2t^l}{\sqrt\beta}\sum_{j=1}^N\tilde X_{t,j}\cdot  e_{k-1}^{N-1}(\tilde X_{t,\{1,\ldots,N\}\setminus\{j\}}^2 )\>
 dB_{t,j}\notag\\
&+\Biggl(  lt^{l-1}\cdot e^N_k(\tilde X_{t}^2)
+2t^l (N-k+\nu+1/(2\beta))(N-k+1)e^N_{k-1}(\tilde X_{t}^2)\Biggr)dt.
\notag\end{align}
\end{lemma}

\begin{proof} As $d[\tilde X_{t,i},\tilde X_{t,j}]= \frac{\delta_{i,j}}{\beta}dt$ by (\ref{SDE-B-normalized}),
It\^{o}'s formula  implies that
\begin{align}d(t^l\cdot e_k^N(\tilde X_{t}^2))= &
lt^{l-1}\cdot e_k^N(\tilde X_{t}^2) \> dt+
t^l\sum_{j=1}^N 2\tilde X_{t,j}\cdot e_{k-1}^{N-1}(\tilde X_{t,\{1,\ldots,N\}\setminus\{j\}}^2)\> d\tilde X_{t,j}
\notag\\
& + \frac{t^l}{\beta}\sum_{j=1}^Ne_{k-1}^{N-1}(\tilde X_{t,\{1,\ldots,N\}\setminus\{j\}}^2)\> dt.\notag
\end{align}
Therefore, by the SDE (\ref{SDE-B-normalized}), 
\begin{align}\label{elementary-symm-b-1}
d(t^l\cdot e^N_k(\tilde X_{t}^2))&=  lt^{l-1}\cdot e^N_k(\tilde X_{t}^2) \> dt+ 
\frac{2t^l}{\sqrt\beta}\sum_{j=1}^N\tilde X_{t,j}\cdot e_{k-1}^{N-1}(\tilde X_{t,\{1,\ldots,N\}\setminus\{j\}}^2 )\>
 dB_{t,j}\notag\\
&+
2t^l\sum_{i,j; i\ne j} \Bigl( 
\frac{  e_{k-1}^{N-1}(\tilde X_{t,\{1,\ldots,N\}\setminus\{j\}}^2)X_{t,j}}{\tilde X_{t,j}-\tilde X_{t,i}}
+\frac{  e_{k-1}^{N-1}(\tilde X_{t,\{1,\ldots,N\}\setminus\{j\}}^2)X_{t,j}}{\tilde X_{t,j}+\tilde X_{t,i}}\Bigr)dt
\notag\\
&+2t^l\nu\sum_{j=1}^N  e_{k-1}^{N-1}(\tilde X_{t,\{1,\ldots,N\}\setminus\{j\}}^2) \> dt
+\frac{t^l}{\beta}\sum_{j=1}^N  e_{k-1}^{N-1}(\tilde X_{t,\{1,\ldots,N\}\setminus\{j\}}^2) \> dt
\notag\\
&= lt^{l-1}\cdot e^N_k(\tilde X_{t}^2) \> dt+ 
\frac{2t^l}{\sqrt\beta}\sum_{j=1}^N\tilde X_{t,j}\cdot e_{k-1}^{N-1}(\tilde X_{t,\{1,\ldots,N\}\setminus\{j\}}^2 )\>
 dB_{t,j}\notag\\
&+2t^l\sum_{i,j; i\ne j} \Bigl( 
\frac{\tilde X_{t,j}^2  e_{k-1}^{N-1}(\tilde X_{t,\{1,\ldots,N\}\setminus\{j\}}^2 )-
\tilde X_{t,i}^2  e_{k-1}^{N-1}(\tilde X_{t,\{1,\ldots,N\}\setminus\{i\}}^2 )}{\tilde X_{t,j}^2-\tilde X_{t,i}^2}
\Bigr)dt
\notag\\
&+2t^l(\nu+1/(2\beta))\sum_{j=1}^N  e_{k-1}^{N-1}(\tilde X_{t,\{1,\ldots,N\}\setminus\{j\}}^2) \> dt.
\end{align}
Simple combinatorial computations (cf. (4.9)-(4.12) in \cite{VW}) show for $i\neq j$ and $k\ge2$ that
\begin{align}\label{elementary-symm-b-2}
\tilde X_{t,j}^2  e_{k-1}^{N-1}&(\tilde X_{t,\{1,\ldots,N\}\setminus\{j\}}^2 )-
\tilde X_{t,i}^2  e_{k-1}^{N-1}(\tilde X_{t,\{1,\ldots,N\}\setminus\{i\}}^2 )\\
&=
 (\tilde X_{t,j}^2-\tilde X_{t,i}^2)\cdot e_{k-1}^{N-2}(\tilde X_{t,\{1,\ldots,N\}\setminus\{i,j\}}^2 ),\notag
\end{align}
\begin{equation}\label{elementary-symm-b-4}
 \sum_{i,j=1,\ldots,N; i\ne j} e_{k-1}^{N-2}(\tilde X_{t,\{1,\ldots,N\}\setminus\{i,j\}}^2 )=
 (N-k+1)(N-k) e_{k-1}^{N}(\tilde X_{t}^2),
\end{equation}
and
\begin{equation}\label{elementary-symm-b-3}
\sum_{j=1}^N  e_{k-1}^{N-1}(\tilde X_{t,\{1,\ldots,N\}\setminus\{j\}}^2)=(N-k+1)e^N_{k-1}(\tilde X_{t}^2).
\end{equation}
 (\ref{elementary-symm-b-2})-(\ref{elementary-symm-b-3}), and (\ref{elementary-symm-b-1})
 now lead to the lemma for $k\ge2$. For $k=1$, the lemma also follows by an even simpler computation.
\end{proof}

We also have the following analogue of Lemma \ref{Brownian-martingale} by the same reasons.
The proof is very similar, hence omitted.
\begin{lemma}\label{Brownian-martingale-b}
For each polynomial $p$ in $N$ variables, $i=1,\ldots,N$, and  $l\ge0$, the process
$\Bigl(\int_0^t s^l\cdot p(\tilde X_{s})\> dB_s^i)_{t\ge0}$
is a martingale.
\end{lemma}

With Lemmas \ref{symmetric-pol-in-t-b}   and \ref{Brownian-martingale-b}, we obtain the following martingales:

\begin{proposition}\label{elementary-symm-martingale-b}
For all $\nu\ge0$, $\beta>0$, $k=1,\ldots,N$, and all starting points $\tilde X_{0,\beta}\in C_N^B$, the process
\begin{align}
\Biggl( e^N_k(\tilde X_{t}^2) + \sum_{l=1}^{k} {(-2t)}^l\binom{N-k+l}{l} (N-k+\nu+1/(2\beta))_l \cdot
 e_{k-l}^N(\tilde X_{t}^2)
\Biggr)_{t\ge0}
\notag\end{align}
(with the Pochhammer symbol $(x)_r:=x(x+1)\cdots(x+r-1)$)
is a martingale. 

\end{proposition}

\begin{proof}  For $\beta\ge 1/2$ and  $\nu\ge1$, 
it follows  readily from  Lemmas \ref{symmetric-pol-in-t-b}   and \ref{Brownian-martingale-b} that
\begin{align}\label{mart-long-b}
\Biggl( e^N_k(\tilde X_{t}^2) +& \sum_{l=1}^{k-1} {(-2t)}^l\binom{N-k+l}{l} (N-k+\nu+1/(2\beta))_l 
 e^N_{k-l}(\tilde X_{t}^2)
\notag \\
& + (-{2})^k\cdot\binom{N}{k}k (N-k+\nu+1/(2\beta))_{k}
  \int_0^t  s^{k-1} \> ds
\Biggr)_{t\ge0}
\notag\end{align}
is  a martingale which proves the claim in this case. The extension to arbitrary  $\beta,\nu$ follows again by 
``analytic continuation'' as in the proof of Proposition \ref{elementary-symm-martingale}.
\end{proof}

Notice that the algebraic functions in Proposition \ref{elementary-symm-martingale-b},
 which lead to martingales, depend on $\nu+1/(2\beta)$ only. Moreover, Lemma \ref{symmetric-pol-in-t-b} and
 Proposition \ref{elementary-symm-martingale-b} remain valid for $\beta=\infty$, in which case the
 SDE (\ref{SDE-B-normalized}) is an ODE, and the process $(\tilde X_{t})_{t\ge0}$ is deterministic
 whenever so is the initial condition for $t=0$. There are several limit theorems 
(laws of large numbers, CLTs) for the limit transition $\beta\to\infty$; see \cite{AKM2}, \cite{AV}, 
 \cite{VW}.   Proposition \ref{elementary-symm-martingale-b} for $\beta\in]0,\infty]$ leads to:

\begin{corollary}\label{constant-expectation-b}
For a fixed   $x_0\in C_N^A$, let $(\tilde X_{t})_{t\ge0}$ be a Bessel process of type B 
with start in  $x_0$ for the parameters  $\nu\ge0$ and $\beta\in]0,\infty]$.
Then for $k=0,1,\ldots,N$ and $t\ge0$,  the expectations
$\mathbb E( e_k^N(\tilde X_{t}^2) )$ depend on $\nu+1/(2\beta)$  only (and not on $\nu,\beta$).
\end{corollary}

\begin{proof} This is clear by 
 Proposition \ref{elementary-symm-martingale-b} and induction.
\end{proof}

We now  start in  $x_0=0\in C_N^A$. Then for $\beta=\infty$, the solution of the
 ODE (\ref{SDE-B-normalized}) can be written via  zeros  of some
  Laguerre polynomial. For this we recapitulate that
for $\alpha>0$,
the Laguerre polynomials  
\begin{equation}\label{def-laguerre}
L_n^{(\alpha)}(x)= \sum_{k=0}^n  \binom{n+\alpha}{n-k} \frac{(-x)^k}{k!}
\end{equation}
 are orthogonal w.r.t.~the density $e^{-x}\cdot x^{\alpha}$ on $]0,\infty[$ as in \cite{S}.
We need the following fact; see \cite{AKM1}, 
 Section 6.7 of \cite{S}, or, in the present notation, \cite{AV}:

\begin{lemma}\label{char-zero-B}
Let $\nu>0$ and denote by $z_1^{(\nu-1)}> \ldots>z_N^{(\nu-1)}>0$ the ordered zeros of  $L_N^{(\nu-1)}$.
Then the vector $y\in C_N^B$ with 
$y^2:=(y_1^2, \ldots, y_N^2)=2(z_1^{(\nu-1)},\ldots, z_N^{(\nu-1)})$
satisfies 
$$\frac{1}{2}y_i=
\sum_{j: j\ne i} \Bigl(\frac{1}{y_i-y_j} +\frac{1}{y_i+y_j}\Bigr) +\frac{\nu}{y_i} \quad\quad (i=1,\ldots,N).$$
\end{lemma}

This leads to the following  solutions of the ODEs (\ref{SDE-B-normalized}) for $\beta=\infty$; cf.~\cite{AV}:

\begin{corollary}\label{special-solution-B1}
Let $\nu>0$ and $y\in C_N^B$ the vector in Lemma \ref{char-zero-B}.
 Then $\phi(t)= \sqrt{t}\cdot y $ is a solution of (\ref{SDE-B-normalized}) for $\beta=\infty$.
\end{corollary}

This result has the following consequence:

\begin{corollary}\label{det-formula-b} Let $( X_{t})_{t\ge0}$ be the Bessel process of type B
starting in 0 with parameters  $\nu\ge 0,\beta>0$.
Then, 
\begin{align} \label{Lag_det_form}\nonumber \mathbb E\bigl(\prod_{i=1}^N (y- X_{t,i}^2)\bigr) &= c_{(\beta\nu,\beta)}^B \int_{\mathbb R^N} \big(\prod_{i=1}^N(y-x_i^2)\big) \cdot t^{-\gamma_{(\beta\nu,\beta)}^B-N/2}\mathrm e^{-\|x\|^2/(2t)} \\ & \hspace{2cm} \times \prod_{i<j}(x_i^2-x_j^2)^{2\beta} \prod_{i=1}^N x_i^{2\beta\nu}\ dx \\ & = (2t \beta)^{N}\cdot(-1)^NN!\cdot L_N^{(\nu+1/(2\beta)-1)}(y/(2t\beta))
  \quad\quad\text{for}\quad y\in \mathbb R.\end{align}
Moreover, for $k=0,1,\ldots,N$,
\begin{equation}\label{det-form-b-2} 
\mathbb E(e^N_k( X_{t}^2))= \binom{N+\nu+1/(2\beta)-1}{k} \cdot \frac{N!}{(N-k)!}\cdot (2t \beta)^{k}.
\end{equation}
\end{corollary}

\begin{proof} As here we need processes with different parameters, we denote the  Bessel processes  and their normalizations with parameters
$\nu,\beta$ by $(X_t(\nu,\beta))_{t\ge 0}$ and $(\tilde X_t(\nu,\beta))_{t\ge 0}$ respectively.
Corollary \ref{constant-expectation-b} and Lemma \ref{special-solution-B1} yield
\begin{align}\label{computation-det-lagu}
\mathbb E\bigl(&\prod_{i=1}^N (y- X_{t,i}(\nu,\beta)^2)\bigr)=
\sum_{k=0}^N (-1)^k \mathbb E\bigl(e^N_k( X_{t}(\nu,\beta)^2)) \cdot y^{N-k}\\
&=\beta^{N}\sum_{k=0}^N (-1)^k \mathbb E\bigl(e^N_k(\tilde  X_{t}(\nu,\beta)^2)) \cdot (y/\beta)^{N-k}\notag\\
&=\beta^{N}\sum_{k=0}^N (-1)^k \mathbb E\bigl(e^N_k(\tilde  X_{t}(\nu+1/(2\beta),\infty)^2)) \cdot (y/\beta)^{N-k}\notag\\
&=\beta^{N} \sum_{k=0}^N  (-1)^k e^N_k(2t\cdot(z_1^{(\nu+1/(2\beta)-1)},\ldots, z_N^{(\nu+1/(2\beta)-1)})) \cdot(y/\beta)^{N-k} \notag\\
&=(2t \beta)^{N} \sum_{k=0}^N  (-1)^k  e^N_k(z_1^{(\nu+1/(2\beta)-1)},\ldots, z_N^{(\nu+1/(2\beta)-1)}) \cdot(y/(2t\beta))^{N-k}\notag\\
&=(2t \beta)^{N}  \prod_{i=1}^N ( y/(2t\beta) - z_i^{(\nu+1/(2\beta)-1)})\notag\\
&
=(2t \beta)^{N}\cdot(-1)^NN!\cdot L_N^{(\nu+1/(2\beta)-1)}(y/(2t\beta)) .
\notag\end{align}
Notice that the last equation follows from the fact that  $L_N^{(\nu+1/(2\beta)-1)}$
 has the leading coefficient $(-1)^N/N!$; see (5.1.8) in \cite{S}. This proves the first statement.
The second statement follows by a comparison of the coefficients in (\ref{computation-det-lagu}) and
 (\ref{def-laguerre}).
\end{proof}

\begin{remark}
Corollary \ref{det-formula-b} can be also applied in the limit case $\beta=0$. In this case, the 
Bessel process is independent of $\nu$ and 
a Brownian motion on $C_N^B$ where the paths are reflected on $\partial C_N^B$. 

Moreover, writing $\mathbb E(e^N_k( X_{t}(\nu,\beta)^2))$ as an integral
  over $C_N^B$ with the explicit densities
of the random variables $X_{t}(\nu,\beta)$ for arbitrary $\beta,\nu\ge0$ as in \cite{RV1}, 
implies by dominated convergence that $\mathbb E(e^N_k( X_{t}^2))$ depends continuously on $\beta,\nu$.
Hence, by  Corollary \ref{det-formula-b}, for $\beta=0$ and
$k=0,1,\ldots,N$,
$$\mathbb E(e^N_k( X_{t}^2))= (2t)^{k} \binom{N}{k}$$.
\end{remark}

\begin{remark}
For the cases $\beta=1/2,1,2$ and start in $0$, Corollary \ref{det-formula-b} admits an interpretation for
Wishart processes and chiral ensembles.
 In particular we obtain Proposition 12 and Corollary 1 of \cite{FG} in the Gaussian cases.
For the details of the connection we refer to Section 3 of \cite{RV2}. 
\end{remark}

\begin{remark}
Corollary \ref{det-formula-b} suggests that it might be interesting to study the expectation of
$\prod_{i=1}^N (y- X_{t,i})$
 for a Dunkl process $( X_{t})_{t\ge0}$ on $\mathbb R^N$ of type $B_{N}$ with  arbitrary multiplicities $(k_1,k_2)$ with $k_1,k_2>0$
 where the process starts in $0\in \mathbb R^N$.
For the theory of Dunkl processes and the notations we refer to \cite{CGY, GY, RV1, RV2}. 
The result is however  simple, as 
the distributions $P_{X_{t}}$ of the $X_t$ are invariant under sign changes in all coordinates separately.
As thus $\mathbb E\bigl(e^N_k( X_{t})\bigr)=0$ for  $k=1,\ldots,N$, we get
\begin{equation}\label{charpol_rm}
\mathbb E\bigl(\prod_{i=1}^N (y- X_{t,i})\bigr)
= \sum_{k=0}^N \mathbb E\bigl(e^N_k( X_{t})\bigr)\cdot y^{N-k}=y^N \quad\quad(y\in \mathbb R).
\end{equation} 
We also note that (\ref{charpol_rm}) coincides with the expectation of the characteristic polynomial of a general random square matrix with independent, centered elements.
 \end{remark}

\begin{remark}
Proposition \ref{elementary-symm-martingale-b} and Corollary \ref {det-formula-b} 
hold also for  Bessel processes $(X_{t}^D)_{t\ge0}$
 of type $D_N$ which live on the
closed Weyl chamber 
$$C_N^D=\{x\in\mathbb R^N: \quad x_1\ge \ldots\ge x_{N-1}\ge |x_N|\},$$
and which depend  on a one-dimensional multiplicity $\beta\ge 0$.
These processes
may be regarded a doubling of the processes $(X_{t}^B)_{t\ge0}$ of type B w.r.t.~the last coordinate
 with the same $\beta$ and $\nu=0$.
More precisely, the squared processes
  $((X_{t}^D)^2)_{t\ge0}$ and  $((X_{t}^B)^2)_{t\ge0}$ are equal in distribution.
This shows that Proposition \ref{elementary-symm-martingale-b} and Corollary \ref {det-formula-b} 
remain valid for  $\nu=0$ there.
\end{remark}

Funding: The first author has been supported by the Deutsche Forschungsgemeinschaft
(DFG) via RTG 2131 High-dimensional Phenomena in Probability - Fluctuations and Discontinuity to visit Dortmund for the preparation of this paper.

\end{document}